\documentclass{amsart}
\usepackage{pgf,amsthm,amscd,amssymb,verbatim,epsf,amsmath,eurosym,amsfonts,mathrsfs,graphicx,tikz-cd,lscape}
\usepackage{amsthm,amscd,amssymb,verbatim,epsf,amsmath,amsfonts,mathrsfs,graphicx,dirtytalk,pdfpages,mdframed,tikz-cd,xy}
\usepackage[colorlinks=true,linkcolor=blue,citecolor=blue]{hyperref}
\usepackage{epigraph}
\begin{document}
\theoremstyle{plain}
\newtheorem{Thm}{Theorem}
\newtheorem{Cor}[Thm]{Corollary}
\newtheorem{Ex}[Thm]{Example}
\newtheorem{Con}[Thm]{Conjecture}
\newtheorem{Main}{Main Theorem}
\newtheorem{Lem}[Thm]{Lemma}
\newtheorem{Prop}[Thm]{Proposition}

\theoremstyle{definition}
\newtheorem{Def}[Thm]{Definition}
\newtheorem{Note}[Thm]{Note}
\newtheorem{Question}[Thm]{Question}

\theoremstyle{remark}
\newtheorem{notation}[Thm]{Notation}
\renewcommand{\thenotation}{}

\errorcontextlines=0
\renewcommand{\rm}{\normalshape}%
\newcommand{\transv}{\mathrel{\text{\tpitchfork}}}
\makeatletter
\newcommand{\tpitchfork}{%
  \vbox{
    \baselineskip\z@skip
    \lineskip-.52ex
    \lineskiplimit\maxdimen
    \m@th
    \ialign{##\crcr\hidewidth\smash{$-$}\hidewidth\crcr$\pitchfork$\crcr}
  }%
}
\makeatother

\title[Evolution to the Bishop Family]{Parabolic evolution with boundary to the Bishop family of holomorphic discs}

\author{Brendan Guilfoyle}
\address{Brendan Guilfoyle\\
          School of Scienec Technology Engineering and Mathematics\\
          Munster Technological University, Kerry\\
          Tralee\\
          Co. Kerry\\
          Ireland.}
\email{brendan.guilfoyle@mtu.ie}
\author{Wilhelm Klingenberg}
\address{Wilhelm Klingenberg\\
 Department of Mathematical Sciences\\
 University of Durham\\
 Durham DH1 3LE\\
 United Kingdom}
\email{wilhelm.klingenberg@durham.ac.uk }

\begin{abstract}
It is proven that a definite graphical rotationally symmetric line congruence evolving under mean curvature flow with respect to the neutral K\"ahler metric in the space of oriented lines of Euclidean 3-space, subject to suitable Dirichlet and Neumann boundary conditions, converges to a maximal surface. When the Neumann condition implemented is that the flowing disc be holomorphic at the boundary, it is proven that the flow converges to a holomorphic disc. 

This is extended to the flow of a family of discs with boundary lying on a fixed rotationally symmetric line congruence, which is shown to converge to a filling by maximal surfaces. Moreover, if the family is required to be holomorphic at the boundary, it is shown that the flow converges to the Bishop filling by holomorphic discs of an isolated complex point of Maslov index 2.
\end{abstract} 

\date{\today}
\maketitle

\section{Motivation and Results}

The motivation of this paper comes from two sources. The first is the well-known Bishop family of holomorphic discs that form a filling in a neighbourhood of an isolated complex point on a real surface in a complex surface \cite{Bishop65} \cite{Elia91} \cite{Forst11} \cite{Mohnke99} \cite{Sukhov08}. The second is to demonstrate how, given the correct geometric setting and boundary conditions, a second order parabolic partial differential equation can converge to a solution of a first order equation. As such, this is a reduced version of the analytic techniques developed in the authors' proofs of conjectures of Toponogov and Carath\'eodory \cite{GK20b} \cite{GK11} . 

In particular, consider the space ${\mathbb L}$ of oriented lines in ${\mathbb R}^3$ endowed with its canonical neutral K\"ahler structure $({\mathbb G},{\mathbb J},\omega)$ \cite{GK05}. A 2-parameter family of oriented lines, classically referred to as a {\it line congruence} \cite{GK04}, forms a surface $\Sigma$ in ${\mathbb L}$, which is Lagrangian with respect to the symplectic form $\omega$ iff the lines are orthogonal to a surface $S$ in ${\mathbb R}^3$. If the surface $\Sigma$ is symplectic then the lines are said to be {\it twisting} and no orthogonal surface exists in ${\mathbb R}^3$. 

In this paper we consider the evolution of a purely twisting rotationally symmetric line congruence in ${\mathbb L}$ by its mean curvature vector with respect to the neutral metric ${\mathbb G}$. Rotationally symmetric reductions of parabolic partial differential equations have been show to yield insights into the general behaviour of such flows, including mean curvature flow \cite{Alt95} \cite{Athan97} \cite{Dzkjk91},  inverse mean curvature flow \cite{Harvie23}, linear curvature flows \cite{GK21}, fully non-linear flows \cite{McCoy14} and Ricci flow \cite{DiG21}. For another evolution seeking a hypersurface foliated by holomorphic discs see \cite{HaK99}.

For the flow to be parabolic we consider an initial surface that has definite induced metric and show that it remains so for all time. The symmetry assumption reduces the problem to that of a curve in two dimensions and we impose one Dirichlet and one Neumann condition on the boundary. The evolution is shown to reduce to the following problem for a real function $\psi$ in one space and one time dimension. 

\vspace{0.1in}

\begin{center}\fbox{\parbox{4.8in}{
\begin{center}{\Large{ \bf I.B.V.P.}}\end{center}
{\it
Let $\theta_0\in(0,\pi/2)$, $C_0\in (0,\infty)$, $C_1\in{\mathbb R}$ and $\psi_0:[0,\theta_0]\rightarrow [0,\infty]$ such that
\[
c_2\leq\frac{\psi_0'}{\sin(2\theta)}\leq c_3 \qquad\qquad for\;\;0\leq\theta\leq\theta_0.
\]
Consider a function $\psi:[0,\theta_0]\times[0,\infty)\rightarrow[0,\infty):(\theta,t)\mapsto\psi(\theta,t)$ satisfying
\[
\frac{\partial\psi}{\partial t}=\frac{\sqrt{\psi}}{\psi'}\psi''-\sqrt{\psi}\cot(2\theta) \qquad\qquad for\;\;0<\theta<\theta_0,
\]
with initial and boundary conditions:
\begin{enumerate}
\item[(i)] $\psi(t=0)=\psi_0$ for $\theta\in[0,\theta_0]$,
\item[(ii)]$\psi(\theta=\theta_0)=C_0$ for $t\geq 0$,
\item[(iii)] $\psi'(\theta=\theta_0)=2C_0\cot(2\theta_0)+C_1$ for $t\geq 0$,
\end{enumerate} 
where a prime denotes differentiation with respect to $\theta$.
\vspace{0.1in}
}
}
}
\end{center}
\vspace{0.1in}
Our main result is:
\begin{Thm}\label{t:1}
There exists $\theta_0,C_0$ and $\psi_0$ such that the solution {\bf I.B.V.P.} converges, as $ t \to \infty $, to a maximal disc in ${\mathbb L}$. 

Moreover, if in the Neumann condition (iii) we have $C_1=0$, then the flow converges  to a holomorphic disc in ${\mathbb L}$ .
\end{Thm}
\vspace{0.1in}
The link to the Bishop family of holomorphic discs arises by allowing $(\theta_0,C_0)=(\theta_0,\psi_0(\theta_0))$ to vary over a line congruence $\tilde{\Sigma}\subset{\mathbb L}$. In particular, fix a purely twisting rotationally symmetric line congruence given by $\tilde{\psi}(\vartheta)$ for $\vartheta\in(0,\vartheta_0]$ and let $\psi_0(\vartheta,\theta)$ be an initial family of disjoint purely twisting rotationally symmetric line congruences whose boundary lies on $\tilde{\Sigma}$. By rotational symmetry, the origin $0\in\tilde{\Sigma}$ is a complex point of the surface, which is assumed to be isolated.

Consider a family of real functions $\psi_\vartheta(\theta,t)$ for $\theta\in[0,\vartheta]$ and $\vartheta\in(0,\vartheta_0]$ satisfying the {\bf I.B.V.P.*} with the same evolution as the {\bf I.B.V.P.} and with boundary conditions:
\begin{enumerate}
\item[(i)*] $\psi_\vartheta(t=0)=\psi_0(\vartheta)$ for $\theta\in[0,\vartheta]$,
\item[(ii)*]$\psi_\vartheta(\theta=\vartheta)=\tilde{\psi}(\vartheta)$ for $t\geq 0$,
\item[(iii)*] $\psi_\vartheta'(\theta=\vartheta)=2{\psi}_\vartheta(\theta=\vartheta)\cot(2\vartheta)+C_1$ for $t\geq 0$.
\end{enumerate}

\begin{Thm}\label{t:2}
Under the {\bf I.B.V.P.*} the initial family $\psi_0$ of discs converges to a family of maximal discs with boundary lying on $\tilde{\Sigma}$. Moreover, if in the Neumann condition (iii)* we have $C_1=0$, then the family converges to the Bishop family of holomorphic discs in ${\mathbb L}$ associated to the complex point $0\in\tilde{\Sigma}$.
\end{Thm}
\vspace{0.1in}
In the next section the full details of the geometric setting are given. Section \ref{s:3} formulates mean curvature flow of purely twisting rotationally symmetric line congruences with boundary conditions in ${\mathbb L}$ as the {\bf I.B.V.P.} above. The proofs of convergence of the flow in Theorems \ref{t:1} and \ref{t:2} are given in Section \ref{s:4}.

\vspace{0.1in}
\section{Geometric Background}\label{s:2}

The space ${\mathbb L}$ of oriented (affine) lines of Euclidean ${\mathbb R}^3$ can be identified with the total space $TS^2$ of the tangent bundle to the 2-sphere. The projection $\pi:{\mathbb L}=TS^2\rightarrow S^2$ maps an oriented line to its direction. This smooth 4-manifold admits a canonical neutral K\"ahler structure $({\mathbb G}, {\mathbb J},\omega)$ which is invariant under the action induced on ${\mathbb L}$ by the Euclidean group of motions \cite{GK05}. Here the neutral metric ${\mathbb G}$ is a pseudo-Riemannian metric of signature $(2,2)$.

If we choose standard holomorphic coordinates $\xi$ on $S^2$, these can be supplemented  by holomorphic coordinates $\eta$ in the fibre, so that we obtain local coordinate $(\xi,\eta)$ on ${\mathbb L}$ that are holomorphic with respect to the canonical complex structure ${\mathbb J}$.  

To explicitly link this to Euclidean coordinates $(x^1,x^2,x^3)$ in ${\mathbb R}^3$ we can use the following. The oriented line determined by $(\xi,\eta)$ is given by \cite{GK04}
\begin{equation}\label{e:mt}
x^1+ix^2=\frac{2(\eta-\bar{\eta}\xi^2)+2\xi(1+\xi\bar{\xi})r}{(1+\xi\bar{\xi})^2},
\qquad
x^3=\frac{-2(\eta\bar{\xi}+\bar{\eta}\xi)+(1-\xi^2\bar{\xi}^2)r}{(1+\xi\bar{\xi})^2},
\end{equation}
where $r$ is a unit parameterization along the line and $r=0$ is the point on the line closest to the origin.

In these coordinates the neutral metric takes the form \cite{GK05} 
\begin{equation}\label{e:metric}
{\mathbb{G}}=4(1+\xi\bar{\xi})^{-2}{\mathbb{I}}\mbox{m}\left(d\bar{\eta} d\xi+\frac{2\bar{\xi}\eta}{1+\xi\bar{\xi}}d\xi d\bar{\xi}\right).
\end{equation}

The associated symplectic structure $\omega$ can be obtained by composition of the complex structure and metric to get
\[
\omega=4(1+\xi\bar{\xi})^{-2}{\mathbb{R}}\mbox{e}\left(d\eta\wedge d\bar{\xi}-\frac{2\bar{\xi}\eta}{1+\xi\bar{\xi}}d\xi\wedge d\bar{\xi}\right).
\]
Let $\Sigma$ be a line congruence in ${\mathbb R}^3$, which we view as a surface in ${\mathbb L}$. Such a surface is said to be {\it graphical} if it is a graph of a section of the canonical bundle $\pi:{\mathbb L}\rightarrow S^2$.  Thus a graphical line congruence can be described by a complex function $(\xi,\eta=F(\xi,\bar{\xi}))$. A line congruence $\Sigma$ is said to be {\it Lagrangian} if $\omega_\Sigma=0$. As mentioned earlier, a line congruence is Lagrangian iff there exists a surface in ${\mathbb R}^3$ that is orthogonal to the lines of the congruence.
\vspace{0.1in}
\begin{Def}\label{d:spinco}
For a section $\eta=F(\xi,\bar{\xi})$, introduce the weighted complex slopes of $F$:
\[
\sigma=-\frac{\partial \bar{F}}{\partial \xi} \qquad\qquad
\rho=\varepsilon+i\lambda=(1+\xi\bar{\xi})^2\frac{\partial}{\partial \xi} \left(\frac{F}{(1+\xi\bar{\xi})^{-2}}\right).
\]
The functions $\lambda$ and $\sigma$ are called the {\it twist} and {\it shear} of the underlying family of oriented lines in ${\mathbb R}^3$. A graphical line congruence is Lagrangian iff $\lambda=0$ and {\it holomorphic} iff $\sigma=0$.
\end{Def}
\vspace{0.1in}
While Lagrangian line congruences form the oriented normal lines to a surface in ${\mathbb R}^3$, twisting line congruences arise as foliations of ${\mathbb R}^3$ by lines \cite{Salvai09}.

Note the two identities, which follow from these definitions:
\[
-(1+\xi\bar{\xi})^2\partial\left[\frac{\bar{\sigma}}{(1+\xi\bar{\xi})^2}\right]=\bar{\partial}\rho+\frac{2F}{(1+\xi\bar{\xi})^2},
\]
\begin{equation}\label{e:id2}
{\mathbb I}{\mbox{m}}\;\partial\left\{(1+\xi\bar{\xi})^2\partial\left[\frac{\bar{\sigma}}{(1+\xi\bar{\xi})^2}\right]\right\}
    =\partial\bar{\partial}\lambda+\frac{2\lambda}{(1+\xi\bar{\xi})^2}.
\end{equation}
Note that for short we have written $\partial$ to mean  differentiation with respect to $\xi$.

For the induced metric we have:

\vspace{0.1in}

\begin{Prop}\label{p:detg}
The metric induced on the graph of a section by the K\"ahler metric is given in coordinates ($\xi,\bar{\xi}$) by;
\[
g=\frac{2}{(1+\xi\bar{\xi})^2}\left[\begin{matrix}
i\sigma & -\lambda\\
-\lambda & -i\bar{\sigma}\\
\end{matrix}
\right],
\]
with inverse
\[
g^{-1}=\frac{(1+\xi\bar{\xi})^2}{2(\lambda^2-\sigma\bar{\sigma})}\left[\begin{matrix}
i\bar{\sigma} & -\lambda\\
-\lambda & -i\sigma\\
\end{matrix}
\right].
\]
\end{Prop}
\begin{proof}
This follows from pulling back the neutral metric (\ref{e:metric}) along a local section $\eta=F(\xi,\bar{\xi})$.
\end{proof}

\vspace{0.1in}

\begin{Def}
A surface $\Sigma\subset {\mathbb L}$ is {\it definite} if the induced metric on $\Sigma$ is either positive or negative definite. For a graph, this means that $\lambda^2-|\sigma|^2>0$, and if, in addition, $\lambda>0$ the induced metric is negative definite, while if $\lambda<0$ it is positive definite.
\end{Def}

\vspace{0.1in}

\begin{Prop}\label{p:indmet}
The induced metric on a Lagrangian surface is Lorentz, except at complex points, where it is degenerate. The induced metric on a holomorphic surface is definite, except at complex points, where it is degenerate. 
\end{Prop}
\begin{proof}
By the previous Proposition we see that the determinant of the induced metric is $2(1+\xi\bar{\xi})^{-4}(\lambda^2-\sigma\bar{\sigma})$, and the result follows. 
\end{proof}

\vspace{0.1in}

Let $\Sigma\rightarrow {\mathbb L}$ be an immersed surface and assume that the induced metric on $\Sigma$ is definite, so that for $\gamma\in\Sigma$ we have the orthogonal splitting $T_\gamma {\mathbb L}=T_\gamma \Sigma\oplus N_\gamma\Sigma$. In what follows we omit the subscript $\gamma$.

\vspace{0.1in}

\begin{Prop}\label{p:frames}
If $\Sigma$ is a definite surface given by the graph $\xi\rightarrow(\xi,\eta=F(\xi,\bar{\xi}))$, then the following vector fields form an orthonormal basis for $T{\mathbb L}$ along $\Sigma$:
\[
E_{(1)}=2{\mathbb R}{ e}\left[\alpha_1\left(\frac{\partial}{\partial \xi}+\partial F\frac{\partial}{\partial \eta}
          +\partial \bar{F}\frac{\partial}{\partial \bar{\eta}}    \right) \right],
\]
\[
E_{(2)}=2{\mathbb R}{ e}\left[\alpha_2\left(\frac{\partial}{\partial \xi}+\partial F\frac{\partial}{\partial \eta}
          +\partial \bar{F}\frac{\partial}{\partial \bar{\eta}}    \right) \right],
\]
\[
E_{(3)}=2{\mathbb R}{ e}\left[\alpha_2\left(\frac{\partial}{\partial \xi}
       +(\bar{\partial} \bar{F}-2(F\partial u-\bar{F}\bar{\partial} u))\frac{\partial}{\partial \eta}
          -\partial \bar{F}\frac{\partial}{\partial \bar{\eta}}    \right) \right],
\]
\[
E_{(4)}=2{\mathbb R}{ e}\left[\alpha_1\left(\frac{\partial}{\partial \xi}
       +(\bar{\partial} \bar{F}-2(F\partial u-\bar{F}\bar{\partial} u))\frac{\partial}{\partial \eta}
          -\partial \bar{F}\frac{\partial}{\partial \bar{\eta}}    \right) \right],
\]
for
\[
\alpha_1=\frac{e^{-u-{\scriptstyle \frac{1}{2}}\phi i+{\scriptstyle \frac{1}{4}}\pi i}}
      {\sqrt{2}[-\lambda-|\sigma|]^{\scriptstyle \frac{1}{2}}}
\qquad\qquad
\alpha_2=\frac{e^{-u-{\scriptstyle \frac{1}{2}}\phi i-{\scriptstyle \frac{1}{4}}\pi i}}
      {\sqrt{2}[-\lambda+|\sigma|]^{\scriptstyle \frac{1}{2}}},
\]
where $\bar{\partial}F=-|\sigma|e^{-i\varphi}$ and we have introduced $e^{2u}=4(1+\xi\bar{\xi})^{-2}$. Note that when $|\sigma|=0$, then $\varphi$ is just a gauge freedom for the frame.

Moreover, $\{E_{(1)},E_{(2)}\}$ span $T\Sigma$ and $\{E_{(3)},E_{(4)}\}$ span $N\Sigma$. 
\end{Prop}

\vspace{0.1in}

Using the same notation as above:

\vspace{0.1in}

\begin{Prop}\label{p:dualbasis}
The dual basis of 1-forms is:
\[
\theta^{(1)}={\mathbb I}m\left[(\alpha_1\partial\bar{F}+\bar{\alpha}_1(\bar{\partial}\bar{F}
    -2(F\partial u-\bar{F}\bar{\partial} u)))d\xi-\bar{\alpha}_1d\eta\right]e^{2u},
\]
\[
\theta^{(2)}=\;{\mathbb I}m\left[(\alpha_2\partial\bar{F}+\bar{\alpha}_2(\bar{\partial}\bar{F}
    -2(F\partial u-\bar{F}\bar{\partial} u)))d\xi-\bar{\alpha}_2d\eta\right]e^{2u},
\]
\[
\theta^{(3)}=\;{\mathbb I}m\left[(\alpha_2\partial\bar{F}-\bar{\alpha}_2\partial F)d\xi
       +\bar{\alpha}_2d\eta\right]e^{2u},
\]
\[
\theta^{(4)}={\mathbb I}m\left[(\alpha_1\partial\bar{F}-\bar{\alpha}_1\partial F)d\xi
       +\bar{\alpha}_1d\eta\right]e^{2u}.
\]
\end{Prop}

\vspace{0.1in}

Now consider the Levi-Civita connection $\overline{\nabla}$ associated with ${\mathbb G}$ and for $X,Y\in T\Sigma$ we have the orthogonal splitting
\[
\overline{\nabla}_X Y= \nabla^\parallel_X Y+A(X,Y),
\]
where $A:T\Sigma\times T\Sigma\rightarrow N\Sigma$ is the second fundamental form of the immersed surface $\Sigma$.

\vspace{0.1in}

\begin{Prop}\label{p:2ndff}
The second fundamental form is:
\[
A(e_{(a)},e_{(b)})=2{\mathbb R}{ e}\left[\beta_{ab}\left(\frac{\partial}{\partial \xi}
       +(\bar{\partial} \bar{F}-2(F\partial u-\bar{F}\bar{\partial} u))\frac{\partial}{\partial \eta}
          -\partial \bar{F}\frac{\partial}{\partial \bar{\eta}}    \right) \right],
\]
for $a,b=1,2$, where
\[
\beta_{11}=\left[i\lambda\partial |\sigma|-\sigma\bar{\partial}|\sigma|+i\lambda\partial\lambda-\sigma\bar{\partial}\lambda
    +|\sigma|(|\sigma|+\lambda)(\partial\varphi-ie^{i\varphi}\bar{\partial}\varphi+2i\partial u-2e^{i\varphi}\bar{\partial}u)\right]
\]
\[
\left/\left[2e^{2u+i\varphi}(|\sigma|+\lambda)^2(-|\sigma|+\lambda)\right]\right.,
\]
\[
\beta_{22}=\left[-i\lambda\partial |\sigma|+\sigma\bar{\partial}|\sigma|+i\lambda\partial\lambda-\sigma\bar{\partial}\lambda
    +|\sigma|(|\sigma|-\lambda)(\partial\varphi+ie^{i\varphi}\bar{\partial}\varphi+2i\partial u+2e^{i\varphi}\bar{\partial}u)\right]
\]
\[
\left/\left[2 e^{2u+i\varphi}(|\sigma|-\lambda)^2(-|\sigma|-\lambda)\right]\right.,
\]
\[
\beta_{12}=\left(-|\sigma|\partial |\sigma|+i\lambda e^{i\varphi}\bar{\partial}|\sigma|+\lambda\partial\lambda-i\sigma\bar{\partial}\lambda
    \right)
\]
\[
\left/\left[2e^{2u+i\varphi}(|\sigma|^2-\lambda^2)\sqrt{|\Delta|}\right]\right. ,
\]
where we define $\Delta=\lambda^2-|\sigma|^2$.
\end{Prop}
\begin{proof}
Consider the parallel and perpendicular projection operators $^\parallel P:T{\mathbb L}\rightarrow T\Sigma$ and
$^\perp P:T{\mathbb L}\rightarrow N\Sigma$. These are given in terms of an adapted frame by
\[
^\parallel P_j^k=\delta_j^k-E_{(3)}^k\theta_j^{(3)}-E_{(4)}^k\theta_j^{(4)}
\qquad\qquad
^\perp P_j^k=\delta_j^k-E_{(1)}^k\theta_j^{(1)}-E_{(2)}^k\theta_j^{(2)}.
\]
The parallel projection operator has the following coordinate description:
\begin{align}
^\parallel P_{\bar{\eta}}^\xi=-{\textstyle{\frac{1}{2\Delta}}}\bar{\sigma}
&\qquad^\parallel P_{\bar{\xi}}^\xi=-{\textstyle{\frac{1}{2\Delta}}} (\bar{\partial}\bar{F}+\lambda i)\bar{\sigma},\nonumber\\
^\parallel P_\eta^\xi= -{\textstyle{\frac{1}{2\Delta}}} \lambda i
&\qquad^\parallel P_\xi^\xi= {\textstyle{\frac{1}{2\Delta}}} [(\partial F -2\lambda i)\lambda i -|\sigma|^2], \nonumber\\
^\parallel P_{\bar{\eta}}^\eta={\textstyle{\frac{1}{2\Delta}}} \bar{\sigma}(\partial F-\lambda i)
&\qquad^\parallel P_{\bar{\xi}}^\eta={\textstyle{\frac{1}{2\Delta}}} [-\bar{\sigma}[\partial F\bar{\partial}\bar{F}-|\sigma|^2
                                 -\lambda i(\bar{\partial}\bar{F}-\partial F)]+2\lambda^2],\nonumber\\
^\parallel P_\eta^\eta= -{\textstyle{\frac{1}{2\Delta}}}[\lambda i\partial F +|\sigma|^2]
&\qquad^\parallel P_\xi^\eta= {\textstyle{\frac{1}{2\Delta}}} \lambda i[(\partial F-2\lambda i)\partial F-|\sigma|^2]\nonumber,
\end{align}
while the perpendicular projection operator is
\[
^\perp P_\xi^\xi=\;^\parallel P_\eta^\eta
\qquad\qquad
^\perp P_{\bar{\xi}}^\xi=-\;^\parallel P_{\bar{\xi}}^\xi
\qquad\qquad
^\perp P_\eta^\xi=-\;^\parallel P_\eta^\xi
\qquad\qquad
^\perp P_{\bar{\eta}}^\xi=-\;^\parallel P_{\bar{\eta}}^\xi,
\]
\[
^\perp P_\xi^\eta=-\;^\parallel P_\xi^\eta
\qquad\qquad
^\perp P_{\bar{\xi}}^\eta=-\;^\parallel P_{\bar{\xi}}^\eta
\qquad\qquad
^\perp P_\eta^\eta=\;^\parallel P_\xi^\xi
\qquad\qquad
^\perp P_{\bar{\eta}}^\eta=-\;^\parallel P_{\bar{\eta}}^\eta.
\]
In terms of a frame in which $\{E_{(1)},E_{(2)}\}$ span the tangent space of $\Sigma$, the second fundamental form is
\[
A_{(ab)}^{\;\;\;\;\;j}=\;^\perp P_k^j\;E_{(a)}^l\overline{\nabla}_l\;E_{(b)}^k.
\]
The result follows by direct computation of these quantities using Propositions \ref{p:frames} and \ref{p:dualbasis}.

\end{proof}

\vspace{0.1in}

\begin{Prop}
The mean curvature vector of the surface $\Sigma$ is:
\[
H=2{\mathbb R}{ e}\left[\gamma\left(\frac{\partial}{\partial \xi}
       +(\bar{\partial} \bar{F}-2(F\partial u-\bar{F}\bar{\partial} u))\frac{\partial}{\partial \eta}
          -\partial \bar{F}\frac{\partial}{\partial \bar{\eta}}    \right) \right],
\]
where
\[
\gamma=\left[-\lambda(-i\lambda\partial |\sigma|+\sigma\bar{\partial}|\sigma|)
     -|\sigma|(i\lambda\partial\lambda-\sigma\bar{\partial}\lambda)-|\sigma|(|\sigma|^2-\lambda^2)(\partial\varphi+2i\partial u)\right]
\]
\[
\left/\left[e^{2u+i\varphi}(|\sigma|^2-\lambda^2)^2\right]\right. .
\]
\end{Prop}
\begin{proof}
The mean curvature vector of the surface $\Sigma$ is the trace of the second fundamental form, which is
\[
H^j=A_{(11)}^{\;\;\;\;\;j}+ A_{(22)}^{\;\;\;\;\;j}.
\]
The result follows from computing this with the aid of the previous Proposition.
\end{proof}
\vspace{0.1in}

\begin{Note}
We can also write the mean curvature vector component (see \cite{GK08} for a variational derivation of this formula)
\begin{equation}\label{e:meanc}
H^\xi=\frac{2e^{-2u}}{\sqrt{|\lambda^2-|\sigma|^2|}}\left[
              ie^{-2u}\partial\left(\frac{\bar{\sigma}e^{2u}}{\sqrt{|\lambda^2-|\sigma|^2|}}\right)
       -\bar{\partial}\left(\frac{\lambda}{\sqrt{|\lambda^2-|\sigma|^2|}}\right)\right].
\end{equation}
\end{Note}

\begin{Cor}
A holomorphic graph is maximal: it has vanishing mean curvature.
\end{Cor}
\begin{proof}
This follows from inserting $\sigma=0$ in equation (\ref{e:meanc}).
\end{proof}
\vspace{0.1in}

The neutral metric on ${\mathbb L}$ is invariant under the action induced on lines by Euclidean isometries acting on ${\mathbb R}^3$ \cite{GK05}. Thus, rotation about a line induces an isometry of the neutral metric and, in particular, rotation about the $x^3-$axis is expressed in holomorphic coordinates by $(\xi,\eta)\mapsto(\xi e^{i\alpha},\eta e^{i\alpha})$.

\vspace{0.1in}
\begin{Def}
A graphical line congruence is said to be {\it rotationally symmetric} if it is given by a map $(\theta,\phi)\mapsto (\xi=\tan(\theta/2)e^{i\phi}, \eta=G(\theta)e^{i\phi})$ for some complex function $G$. A rotationally symmetric line congruence will be said to be {\it purely twisting} if $G$ is imaginary.
\end{Def}

\vspace{0.1in}
\begin{Prop}
For a graphical rotationally symmetric line congruence with defining complex function $G$ as above, 
\[
\lambda= {{\frac{i\cos(\theta/2)}{4\sin(\theta/2)}}}\left[ \cos(\theta/2)\sin(\theta/2)(\bar{G}'-G')+(4\sin^2(\theta/2)-1)(G-\bar{G})\right]
\]
\[
\sigma=-{{\frac{\cos(\theta/2)}{2\sin(\theta/2)}}}\left[ \cos(\theta/2)\sin(\theta/2) \bar{G}'-\bar{G}\right]e^{-2i\phi}.
\]
If, moreover, the line congruence is purely twisting so that $G+\bar{G}=0$ one has
\begin{equation}\label{e:lsrs}
\lambda={\textstyle{\frac{1}{2\sqrt{\psi}}}}(\psi'+2\cot\theta\;\psi) \qquad \sigma={\textstyle{\frac{i}{2\sqrt{\psi}}}}(\psi'-2\cot\theta\;\psi)e^{-2i\phi},
\end{equation}
where we have introduced the positive function $\psi$ defined by
\[
\eta=F=Ge^{i\phi}=\frac{i\sqrt{\psi(\theta)}}{\cos^2(\theta/2)}e^{i\phi}.
\]
\end{Prop}
\begin{proof}
    The first two equations follow from the expressions for the twist and shear in Definition \ref{d:spinco} and the transformation from holomorphic to polar coordinates $(\xi,\bar{\xi})\mapsto(\theta,\phi)$ for which
    \begin{equation}\label{e:dcoords}
    \frac{\partial}{\partial \xi}=\cos^2(\theta/2)\left(\frac{\partial}{\partial\theta}-\frac{i}{2\sin(\theta/2)\cos(\theta/2)}\frac{\partial}{\partial\phi}\right)e^{-i\phi},
    \end{equation}
    while the second two follow from substitution of the purely twisting condition.
\end{proof}
\vspace{0.1in}
\begin{Prop}
    For a purely twisting rotationally symmetric line congruence, the real function $\psi$ above is one quarter of the distance-squared of the oriented line to the origin.
\end{Prop}
\begin{proof}
This follows from the fact that the distance $d$ to the origin of an oriented line with holomorphic coordinates $(\xi,\eta)$ is
(see equation (\ref{e:mt}) with $r=0$)
\[
x^1_0+ix^2_0=-\frac{2(\eta-\bar{\eta}\xi^2)}{(1+\xi\bar{\xi})^2},
\qquad
x^3_0=-\frac{2(\eta\bar{\xi}+\bar{\eta}\xi)}{(1+\xi\bar{\xi})^2},
\]
and so the square of the perpendicular distance to the origin is
\[
\chi^2= (x^1_0)^2+(x^2_0)^2+(x^3_0)^2 =\frac{4\eta\bar{\eta}}{(1+\xi\bar{\xi})^2}=4\psi.
\] 
\end{proof}
\vspace{0.1in}
Note that $\lambda^2-|\sigma|^2=2\cot\theta\psi'$ and so by the proof of Proposition \ref{p:detg} a purely twisting congruence is (negative) definite if $\psi'>0$.

\vspace{0.1in}
\begin{Prop} {\label{hol:1}}
    The only non-singular purely twisting rotationally symmetric maximal graphs have
    \[
    \psi=a+b\cos(2\theta),
    \]
    for $a,b\in{\mathbb R}$. The only pure twisting rotationally symmetric holomorphic graphs are those with $a=-b$.
\end{Prop}
\begin{proof}
    By Theorem 3 of \cite{GK08}, the only rotationally symmetric maximal graphs are
    \[
    \eta=\left(A_1R+B_1R^{-1}(1+R^2)^2\pm i\sqrt{A_2R^2+B_2(1+R^2)^2-B_1^2R^{-2}(1+R^2)^4}\right)e^{i\phi},
    \]
    where $A_1,A_2,B_1,B_2\in{\mathbb R}$ and $R=\tan(\theta/2)$. To be pure twisting we must have $A_1=0$ and $B_1=0$. Then changing coordinates from $R$ to $\theta$ yields the stated result.
\end{proof}
\vspace{0.1in}
\section{The Parabolic Evolution}\label{s:3}
We now investigate the initial boundary value problem, namely unparameterised mean curvature flow with boundary conditions. In particular we consider a family of definite sections $f_t:D\rightarrow {\mathbb L}$ such that
\[
\frac{\partial f}{\partial t}^\bot=H,
\]
where $H$ is the mean curvature vector of $f_t(D)$ and $\bot$ is projection perpendicular to $f_t(D)$ with respect to ${\mathbb G}$.

Consider the evolution when the flowing surface is a graph of a section of ${\mathbb L}\rightarrow S^2$. In this case it is most convenient to use the base to parameterize the surfaces. That is, we consider a flowing surface given by $\xi\mapsto(\xi,\eta=F_t(\xi,\bar{\xi}))$. 

We compute the explicit expressions for the flow of the complex function $F_t$.
\vspace{0.1in}
\begin{Prop}\label{p:graphflow}
For a definite graph in ${\mathbb L}$, the mean curvature flow is
\begin{align}
\frac{\partial F}{\partial t}=&g^{jk}\partial_j\partial_k F+\frac{i\bar{\sigma}}{\Delta}\left((\sigma\xi-\bar{\rho}\bar{\xi})
    (1+\xi\bar{\xi})+\bar{F}-\bar{\xi}^2F\right)\nonumber\\
&=\frac{(1+\xi\bar{\xi})^2}{2(\lambda^2-\sigma\bar{\sigma})}\left(-2\bar{\sigma}\partial\lambda-i\bar{\sigma}\bar{\partial}\sigma+2\lambda\partial\bar{\sigma}+i\sigma\bar{\partial}\bar{\sigma}
   +\frac{4i\bar{\sigma}(\sigma\xi+\lambda i\bar{\xi})}{1+\xi\bar{\xi}}\right)\nonumber.
\end{align}
\end{Prop}
\begin{proof}
Consider a definite surface $f:\Sigma\times[0,t_0)\rightarrow{\mathbb L}$ such that 
$f_t(\xi,\bar{\xi})=(\xi,F_t(\xi,\bar{\xi}))$. Then
\[
\frac{\partial f}{\partial t}=\frac{\partial F}{\partial t}\frac{\partial }{\partial \eta}+\frac{\partial \bar{F}}{\partial t}\frac{\partial }{\partial \bar{\eta}}.
\]
Projecting onto the normal of $\Sigma$
\begin{align}
\frac{\partial f}{\partial t}^\bot=&(^\bot P^\xi_\eta \dot{F}+\;^\bot P^\xi_{\bar{\eta}} \dot{\bar{F}})\frac{\partial }{\partial \xi}
    +(^\bot P^\eta_\eta \dot{F}+\;^\bot P^\eta_{\bar{\eta}} \dot{\bar{F}})\frac{\partial }{\partial \eta}\nonumber\\
&\qquad+(^\bot P^{\bar{\xi}}_{\bar{\eta}} \dot{\bar{F}}+\;^\bot P^{\bar{\xi}}_{\eta} \dot{F})\frac{\partial }{\partial \bar{\xi}}
    +(^\bot P^{\bar{\eta}}_{\bar{\eta}} \dot{\bar{F}}+\;^\bot P^{\bar{\eta}}_{\eta} \dot{F})\frac{\partial }{\partial \bar{\eta}}\nonumber,
\end{align}
and so the mean curvature flow is
\[
^\bot P^\xi_\eta \dot{F}+\;^\bot P^\xi_{\bar{\eta}} \dot{\bar{F}}=H^\xi,
\]
or from the expressions of the projection operators given in the proof of Proposition \ref{p:2ndff}
\[
\frac{\lambda i}{2\Delta}\dot{F}-\frac{\bar{\sigma}}{2\Delta} \dot{\bar{F}}=H^\xi.
\]
Combining this with its complex conjugate we have
\begin{equation}\label{e:fdot1}
\dot{F}=-2\lambda iH^\xi+2\bar{\sigma}H^{\bar{\xi}}.
\end{equation}
Using the expression (\ref{e:meanc}) for the mean curvature we get that
\begin{align}
H^\xi=&\frac{(1+\xi\bar{\xi})^2}{4\Delta^2}\Big[2\left(i\partial\bar{\sigma}-\bar{\partial}\lambda
    -\frac{2i\bar{\xi}\bar{\sigma}}{1+\xi\bar{\xi}}\right)\Delta\nonumber\\
  &\qquad\qquad\qquad-2i\lambda\bar{\sigma}\partial\lambda+i\sigma\bar{\sigma}\partial\bar{\sigma}+i\bar{\sigma}^2\partial \sigma+2\lambda^2\bar{\partial}\lambda
      -\lambda\sigma\bar{\partial}\bar{\sigma}-\lambda\bar{\sigma}\bar{\partial}\sigma\Big]\nonumber,
\end{align}
and the second equality stated in the Proposition follows from inserting this in equation (\ref{e:fdot1}).

To see that the first equality in the Proposition holds, compute
\begin{align}
g^{jk}\partial_j\partial_k F&=\frac{(1+\xi\bar{\xi})^2}{2\Delta}\left(i\bar{\sigma}\partial^2F-2\lambda\partial\bar{\partial}F
     -i\sigma\bar{\partial}^2F \right)\nonumber\\
&=\frac{(1+\xi\bar{\xi})^2}{2\Delta}\left[i\bar{\sigma}\partial\left(\theta+i\lambda+\frac{2\bar{\xi}F}{1+\xi\bar{\xi}}\right)
    +2\lambda\partial\bar{\sigma}+i\sigma\bar{\partial}\bar{\sigma} \right]\nonumber\\
&=\frac{(1+\xi\bar{\xi})^2}{2\Delta}\left[-2\bar{\sigma}\partial\lambda-i\bar{\sigma}\bar{\partial}\sigma+i\sigma\bar{\partial}\bar{\sigma} 
   +2\lambda\partial\bar{\sigma}\right.\nonumber\\
&\qquad\qquad\qquad\qquad\left.+i\bar{\sigma}\left(\frac{2(\sigma\xi+\rho\bar{\xi})}{1+\xi\bar{\xi}}
   -\frac{2(\bar{F}-\bar{\xi}^2F)}{(1+\xi\bar{\xi})^2} \right)\right]\nonumber,
\end{align}
where we have used identity (\ref{e:id2}) in the more convenient form
\[
\partial\theta=i\partial\lambda-(1+\xi\bar{\xi})^2\partial\left(\frac{\bar{\sigma}}{(1+\xi\bar{\xi})^2}\right)-\frac{2F}{(1+\xi\bar{\xi})^2}.
\]
Thus 
\[
g^{jk}\partial_j\partial_k F+\frac{i\bar{\sigma}}{\Delta}\left((\sigma\xi-\bar{\rho}\bar{\xi})
    (1+\xi\bar{\xi})+\bar{F}-\bar{\xi}^2F\right)\qquad\qquad\qquad\qquad\qquad\qquad\qquad\qquad
\]
\[
=\frac{(1+\xi\bar{\xi})^2}{2(\lambda^2-\sigma\bar{\sigma})}\left(-2\bar{\sigma}\partial\lambda-i\bar{\sigma}\bar{\partial}\sigma
    +i\sigma\bar{\partial}\bar{\sigma}+2\lambda\partial\bar{\sigma}+\frac{4i\bar{\sigma}(\sigma\xi+\lambda i\bar{\xi})}{1+\xi\bar{\xi}}\right),
\]
as claimed.
\end{proof}
\vspace{0.1in}

 We now reduce the above equations in the case considered in this paper.
\vspace{0.1in}
\begin{Prop}
    Mean curvature flow for a purely twisting rotationally symmetric graph reduces to the following single equation for the real function $\psi(\theta)$:
    \[
\frac{\partial\psi}{\partial t}=\frac{\sqrt{\psi}}{\psi'}\psi''-\sqrt{\psi}\cot(2\theta), 
\]
where a prime represents differentiation with respect to $\theta$.
\end{Prop}
\begin{proof}
    This follows from the expressions given in equations (\ref{e:lsrs}) for the twist and shear and equation (\ref{e:dcoords}) for the derivarives into those of Proposition \ref{p:graphflow}. 
    
    In particular, note that the flow preserves the purely twisting condition and therefore mean curvature flow reduces from a system to a single equation.
\end{proof}
\vspace{0.1in}
The boundary conditions we consider are the Dirichlet condition (ii) $\psi(\theta_0)=C_0$ for some $\theta_0\in(0,\pi/2)$, $C_0\in{\mathbb R}$ and the Neumann condition (iii) that
\[
\psi'(\theta_0)=2C_0\cot(2\theta_0)+C_1, 
\]
for some $C_1\in{\mathbb R}$. Note that by the second of equations (\ref{e:lsrs}) if $C_1=0$  the disc is holomorphic at $\theta=\theta_0$.

\vspace{0.1in}
\section{Proofs of the Main Theorems}\label{s:4}

\vspace{0.1in}
\begin{Prop}\label{p:gradest}
    Under mean curvature flow, a definite purely twisting rotationally symmetric graph in ${\mathbb L}$ satisfies the following a prior\'i estimates
    \[
    0 < C_4(\psi_0)\leq\psi\leq C_5(\psi_0)
    \]
    \[
    0 < C_2(\psi_0)\leq\frac{\psi'}{\sin(2\theta)}\leq C_3(\psi_0).
    \] 
\end{Prop}
\begin{proof}
    This follows from the fact that at an interior turning point of the quantity ${\textstyle{\frac{\psi'}{\sin(2\theta)}}}$ we have
    \[
    \frac{\partial}{\partial t}\left(\frac{\psi'}{\sin(2\theta)}\right)=\frac{\sqrt{\psi}}{\psi'}\left(\frac{\psi'}{\sin(2\theta)}\right)'',
    \]
    and so by the maximum principle, the second estimate holds for interior turning points. 

    By the Dirichlet and Neumann conditions the estimates also hold at the boundary.
\end{proof}
\vspace{0.1in}

In the case where the Neumann condition is holomorphicity along the boundary, i.e. $C_1=0$ in conditions (iii) and (iii)*, the flowing graph is asymptotically holomorphic.
\vspace{0.1in}
\begin{Prop}\label{p:ashol}
    Under mean curvature flow with holomorphic boundary condition $C_1=0$, a definite purely twisting rotationally symmetric graph in ${\mathbb L}$ has shear satisfying
    \[
   |\sigma|\leq \frac{C_6(\theta_0,\psi_0)}{t} \qquad\qquad for\;\; t>0.
    \]
   \end{Prop}
\begin{proof}
    This follows from the fact that at an interior turning point of  $|\sigma|=(\sqrt{\psi})'-\cot\theta\sqrt{\psi}$ we have
    \[
    \frac{\partial}{\partial t}|\sigma|=\frac{\sqrt{\psi}}{\psi'}|\sigma|''-\frac{|\sigma|^2}{2\psi}+\frac{a_3(\psi')^3+a_2(\psi')^2\psi+a_1(\psi')\psi^2+a_0\psi^3}{2\sin^4\theta\cos^2\theta\psi(\psi')^2},
    \]
    where
    \[
    a_3=-\cos^3\theta\sin^3\theta \qquad a_2=\sin^2\theta(2\sin^4\theta-1),
    \]
    and
    \[
    a_1=8\cos^3\theta\sin\theta \qquad a_0=-\cos^2\theta.
    \]
    This means that for small $\theta_0$ and $\psi'$
    \[
    \frac{\partial}{\partial t}|\sigma|\leq\frac{\sqrt{\psi}}{\psi'}|\sigma|''-\frac{|\sigma|^2}{2\psi}.
    \]
    Thus, if the boundary is holomorphic, by the maximum principle and ODE comparison the estimate holds (see e.g. Lemma 4.5 of \cite{EaH91}). 
\end{proof}
\vspace{0.1in}
\begin{Prop}\label{p:secest}
    Under mean curvature flow, a definite purely twisting rotationally symmetric graph in ${\mathbb L}$ satisfies the following a prior\'i estimate
    \[
    C_6(\psi_0,\theta_0)\leq\frac{\psi''}{\sin(2\theta)(\psi')^{2/3}}\leq C_3(\psi_0,\theta_0).
    \] 
\end{Prop}
\begin{proof}
    At an interior turning point of the quantity 
    \[
    B=\frac{\psi''}{\sin(2\theta)(\psi')^{2/3}},
    \]
    we have
    \begin{equation}\label{e:sdest}
    \frac{\partial B}{\partial t}=\frac{\sqrt{\psi}}{\psi'}B''+a(\psi'')^2+b\psi''+c,
    \end{equation}
    where 
    \[
    a=\frac{-28\cos(2\theta)\psi-\sin(2\theta)\psi'}{6\sin^2(2\theta)\sqrt{\psi}(\psi')^{8/3}}\leq0
    \]
    \[
    b=\frac{-3\sin(2\theta)(\psi')^2+20\cos(2\theta)\psi\psi'-(80+32\cot^2(2\theta))\sin(2\theta)\psi^2}{12\sin^2(2\theta)\sqrt{\psi}\psi(\psi')^{8/3}}
    \]
    \[
    c=\frac{\cot(2\theta)(\psi')^2+8(1+\cot^2(2\theta))\psi\psi'-32\cot(2\theta)(1+\cot^2(2\theta))\psi^2}{2\sin(2\theta)\sqrt{\psi}\psi(\psi')^{2/3}}.
    \]
    Now a calculations shows that
    \[
    \lim_{\theta\rightarrow 0}\sin^6\theta(b^2-4ac)=-\frac{41\psi}{9(\psi')^{10/3}}\leq 0.
    \]
    Thus, for small enough $\theta_0$ and $\psi_0$, the quadratic terms in $\psi''$ in equation (\ref{e:sdest}) are non-positive and the estimate follows from the maximum principle in the interior.

    For the boundary, differentiate the Dirichlet condition with respect to time to conclude that
    \[
    \psi''(\theta_0)=\cot(2\theta_0)\psi'(\theta_0)
    \]
    which is bounds the second derivative.    
\end{proof}
\vspace{0.1in}
\noindent{\bf Proof of Theorem \ref{t:1}}:

\vspace{0.1in}
By Proposition \ref{p:gradest} and since $\theta_0 < \pi/2$, the evolution  {\bf I.B.V.P.} is strictly parabolic for as long as the solution $\psi(\cdot , t)$ exists and is smooth, say for $ t \in [0,T)$.  Again by the bounds in Proposition \ref{p:gradest}, the Nash-Moser regularity gives uniform $C^{1,\alpha}$ bounds over the open maximal time interval. This allows to extend the solution to $t = T$ in $C^{1,\alpha}$. Smoothness now follows from Proposition \ref{p:secest}, and allows to repeat the argument to give a smooth solution on $[0, T = \infty)$.

Finally by Proposition \ref{p:ashol}, we have a holomorphic limit in case $C_1 = 0$. 
\qed \\
\vspace{0.1in}

Explicitly, the flow converges to the purely twisting congruence $\psi=a+b\cos(2\theta)$, which can be written in terms of the initial and boundary conditions as: 
\[
a=\frac{-C_1\cos(2\theta_0)+2\cot(\theta_0)\psi_0}{2\cot(\theta_0)(1-\cos(2\theta_0))} \qquad \qquad 
b=\frac{C_1-2\cot(\theta_0)\psi_0}{2\cot(\theta_0)(1-\cos(2\theta_0))}.
\]
\vspace{0.1in}

\noindent{\bf Proof of Theorem \ref{t:2}}:

\vspace{0.1in}
Fix a purely twisting rotationally symmetric graphical line congruence $\tilde{\Sigma}$ with an isolated complex point at $0\in\tilde{\Sigma}$. Now choose an initial family of disjoint purely twisting rotationally symmetric graphical line congruences whose boundary lies on $\tilde{\Sigma}$, parameterized by the angle $\theta$ at which they intersect $\tilde{\Sigma}$. 

The proof then is the same as that of Theorem \ref{t:1} for each leaf of the foliation individually. That they remain disjoint follows from the parabolic maximum principle and the flow converges to a foliation by maximal surfaces. In the case of holomorphic boundary condition $C_1=0$, the foliation is by holomorphic discs, namely the Bishop family of the isolated complex point $0\in\tilde{\Sigma}$.
\qed

\end{document}